\documentclass[12pt]{article}

\usepackage{amsmath,amsthm,amssymb,graphicx}

\def\N{\mathbb N}
\def\N0{\mathbb N_0}

\def\real{\mathbb R}
\def\complex{\mathbb C}

\def\Berg{{\cal A}^2}
\def\Hardy{{\mathcal H}}
\def\disk{\mathbb D}

\def\Dalg{\mathbb A _{\beta }}

\def\Dreg{\disk -regular}
\def\Rsym{R-symmetric}

\def\dop{\mathcal L}
\def\dopmin{{\mathcal L}_{min}}
\def\dom{\mathcal D}
\def\dommin{{\mathcal D}_{min}}

\title{\bf Analytic Differential Operators \\
on the Unit Disk
}
\author{Robert Carlson \\
Department of Mathematics \\ 
University of Colorado at Colorado Springs \\
rcarlson@uccs.edu \\
719-255-3561 \\
ORCID ID 0000-0002-6733-7922}

\newtheorem{thm}{Theorem}[section]
\newtheorem{cor}[thm]{Corollary}
\newtheorem{lem}[thm]{Lemma}
\newtheorem{prop}[thm]{Proposition}

\theoremstyle{definition}

\theoremstyle{remark}

\newcommand{\thmref}[1]{Theorem~\ref{#1}}

\newcommand{\lemref}[1]{Lemma~\ref{#1}}
\newcommand{\corref}[1]{Corollary~\ref{#1}}
\newcommand{\propref}[1]{Proposition~\ref{#1}}

 \numberwithin{equation}{section}

\begin{document}

\maketitle

\begin{abstract}
Formally symmetric differential operators on weighted Hardy-Hilbert spaces are analyzed, along with adjoint pairs of differential operators.
Eigenvalue problems for such operators are rather special, but include many of the classical Riemann and Heun equations.
Symmetric minimal operators are characterized.  
A regular class whose leading coefficients have no zeros on the unit circle are shown to be essentially self-adjoint.
Eigenvalue asymptotics are established.  Some extensions to non-self-adjoint operators are also considered.
 
\end{abstract}

\vskip 50pt


\vskip 10pt


\newpage

\tableofcontents

\newpage

\section{Introduction}

If 
\begin{equation} \label{dexpress}
L = \sum_{k=0}^N p_k(z)D^k, \quad D = \frac{d}{dz},
\end{equation}
is a differential expression whose coefficients $p_k(z)$ are analytic on the unit disk $\disk \subset \complex$,
it is natural to ask about the action of $L$ as an operator on a Hilbert space or Banach space of analytic functions on $\disk$.
Because self adjoint operators play such a dominant role in applications, their identification and analysis is likely to be
fundamental.  A. Villone initiated just such a project in his dissertation, written under the direction of E.A. Coddington,
and in subsequent papers \cite{Villone1} - \cite{Villone5}. 
The chosen Hilbert space was the Bergman space $\Berg $ of analytic functions square integrable with respect to area measure on $\disk$. 
This project was extended in the dissertation of W. Stork, written under the direction of J. Weidman, as reported in \cite{StorkA, StorkB}. 

Both Villone and Stork begin with a minimal operator $\dopmin$ on $\Berg$ acting by $\dopmin f = Lf$ 
on the domain $\dommin$ consisting of polynomials in $z$.  Basic questions include whether $\dopmin$ is symmetric, 
has self adjoint extensions, or is essentially self adjoint.  While the framework appears natural, the class of
symmetric operators $\dopmin$ is quite small.  The coefficients $p_k(z)$ must be polynomials of degree at most $N+k$.
First order minimal operators are symmetric if and only if
\[L = (a_2z^2 + a_1z + a_0)D + (b_1z + b_0), \]
with constants $a_j,b_j$ satisfying
\[a_0 = \overline{a_2}, \quad a_1,b_0 \in \real , \quad b_1 = 2a_2.\]

This work starts by generalizing the earlier investigations in two ways: by introducing new Hilbert spaces, and by 
extending the focus on symmetric operators to include adjoint pairs of differential operators.
The single space $\Berg $ is extended to a more general class of weighted Hardy spaces \cite{Shields} in which the 
functions $z^n$, $n=0,1,2,\dots $ are orthogonal.  The adjoint pairs of differential expressions $L,L^+$ on a weighted
Hardy space are still highly constrained; again, the coefficients $p_k(z)$ must be polynomials of degree at most $N+k$.
In turn, asking for adjoint pairs $L,L^+$ of the same order (a natural requirement when symmetry is a central concern)
constrains the admissible weights, which are essentially characterized by a single positive parameter.    
Adjoint pairs $L,L^+$ have a role in the formally symmetric constructions $L+L^+$, $L^+L$, and formally symmetric operator matrices such as
\[\begin{pmatrix} 0 & L^+ \cr L & 0  \end{pmatrix}.\]
Expressions $L^+L$ and their pertubations will appear as generators of semigroups \cite{CPart,GGG}.
By exploiting the $L + L^+$ form, symmetric minimal operators $\dopmin$ are given a straightforward characterization
that was missing in the earlier work.  

The next section introduces a class of $\Dreg$ expressions.
By focusing on symmetric differential expressions whose leading coefficients had no
zeros on the unit circle, Stork \cite{StorkB} identified a class of essentially self-adjoint minimal operators $\dopmin$. 
The self-adjoint extensions have discrete spectrum; a general description of the eigenvalue distribution was provided.   
Leaning rather heavily on the operator-theoretic methods of \cite{Kato},
new techniques are developed below to extend Stork's result to symmetric $\Dreg$ expressions on the weighted Hardy spaces, 
with related results for non-self-adjoint operators.   A simple formula for the Fredholm index is established.

Continuing the perturbation analysis, the last section provides significantly improved eigenvalue estimates for self-adjoint $\Dreg $ operators and their
non-self-adjoint perturbations by lower order terms.  Initial steps involve restricting the expressions to the unit
circle, applying conventional reductions, and studying the eigenvalue problem for an extended operator in a conventional Hilbert space of
$2\pi$-periodic functions on $\real $.  Once the eigenvalues are constrained, final results  are developed 
with an operator deformation argument in the weighted Hardy spaces.

\section{Adjoint pairs in weighted Hardy spaces}
 
\subsection{Weighted Hardy spaces}
 
A sequence $\beta = \{ \beta _n, n = 0,1,2,\dots \} $ of positive numbers can be used to define an inner product
on the vector space $\dommin$ of polynomials $f:\complex \to \complex $.
For polynomials $f_1 = \sum_{n=0}^{\infty } a_nz^n$ and $f_2 = \sum_{n=0}^{\infty } c_nz^n$ written as power series
whose coefficients are eventually zero, an inner product and corresponding norm are given by 
\[\langle f_1,f_2 \rangle _{\beta} = \sum_{n=0}^{\infty} a_n \overline{c_n} \beta _n^2, \quad \| f_1 \| _{\beta }^2 = \sum_{n=0}^{\infty } |a_n|^2 \beta _n^2.\]
The completion of this inner product space is the Hilbert space $\Hardy _{\beta }$,
consisting of the formal power series $g = \sum_{n=0}^{\infty} a_nz^n$ such that
\[\| g \| _{\beta }^2 = \sum_{n=0}^{\infty} |a_n|^2 \beta _n^2 < \infty .\]
$\Hardy _{\beta }$ has an orthonormal basis given by 
\begin{equation} \label{orthobasis} 
e_n = z^n/\beta _n, \quad n = 0,1,2,\dots .
\end{equation}
Standard examples of weighted spaces $\Hardy _{\beta }$ include the classical Hardy space $\Hardy ^2$, 
where, if $z = |z|e^{i\theta }$, the inner product is  
 \begin{equation} \label{Hardyip}
\langle f,g \rangle _{\Hardy ^2} = \frac{1}{2\pi} \int_0^{2\pi} f(e^{i\theta }) \overline{g(e^{i\theta }) } \ d \theta , \quad \beta _n = 1,
\end{equation}
and the Bergman space $\Berg $, with 
\begin{equation} \label{Bergip}
\langle f,g \rangle _{\Berg} = \int_{\disk } f(z)\overline{g(z)} \ dx \ dy, \quad \beta _n =  \sqrt{\pi / (n+1)} .
\end{equation}

If multiplication by $z$, and so polynomials in $z$, act as bounded operators on $\Hardy _{\beta }$, 
then differential expressions $L$ with coefficients in $\Hardy _\beta $ will satisfy $Lf \in \Hardy _{\beta}$ for all $f \in \dommin $.  
That is, $\dopmin$ will be densely defined. The needed weight condition is easily characterized.

\begin{prop} \label{bndop}
The operator $M_zf = zf$ is bounded on $\Hardy _{\beta }$ if and only if the sequence $\{ \beta _{n+1}/\beta _n \}$ is bounded, in which case
$\| M_z \| \le  \sup _n \beta _{n+1}/\beta _n $.
\end{prop}

\begin{proof}
If $f(z) = \sum_{n=0}^{\infty} a_nz^n$, then 
\[\| zf \| _{\beta }^2 = \| \sum_{n=0}^{\infty} a_{n}z^{n+1} \| _{\beta } ^2 = \sum_{n=0}^{\infty} |a_{n} |^2 \beta _{n}^2 \frac{\beta _{n+1}^2}{\beta _{n}^2}. \]
If 
\begin{equation} \label{multbnd}
C_{\beta } = \sup _n\beta _{n+1}/\beta _n < \infty,
\end{equation}
then $\| zf \| _{\beta } \le C_{\beta} \| f \| _{\beta } $.
If $\beta _{n+1}/\beta _n$ is unbounded, then 
\[\| z^n \| _{\beta } = \beta _n , \quad \| z \cdot z^n \| _{\beta} = | \frac{\beta _{n+1}}{\beta _n} | |\beta _n| , \]
and $M_z$ is unbounded. 
\end{proof}

When $C _{\beta } < \infty $ a simple induction shows that $\beta _n \le \beta _0 C_{\beta }^n$, implying that
the power series of elements of $\Hardy _{\beta }$ converge to an analytic function on some disk centered at $z=0$.
The next proposition shows that the condition 
\begin{equation} \label{betabnd}
\lim_{n \to \infty}  \beta _n^{1/n} = r, \quad r > 0,
\end{equation}
implies that $\Hardy _{\beta }$ has a natural interpretation as a space of analytic functions on $\disk _r$,
the open disk of radius $r$ centered at $z=0$.

\begin{prop} \label{bndpe}
If \eqref{betabnd} holds, then every $f \in \Hardy _{\beta }$ is analytic in the open disk $\disk _r$.  For $j = 0,1,2,\dots $ and $z_1 \in \disk _r$,
the linear functional $f^{(j)}(z_1)$ given by derivative evaluation is uniformly bounded in the $\Hardy _{\beta}$ norm on compact subsets of $\disk _r$.
For any $R > r$, every function analytic in $\disk _{R} $ belongs to $\Hardy _{\beta }$,
and some functions in $\Hardy _{\beta }$ do not extend analytically to $\disk _R$.
\end{prop}

\begin{proof}
Suppose $f = \sum_{n=0}^{\infty} a_nz^n \in \Hardy _{\beta}$.  For $|z_1| < r$, the Cauchy-Schwarz inequality gives
\[|f^{(j)}(z_1)| \le \sum_{n=j}^{\infty} |a_n| n(n-1) \dots (n-j+1)|z_1^{n-j}| \]
\[ \le \Bigl ( \sum_{n=j}^{\infty} |a_n|^2\beta _n^2  \Bigr ) ^{1/2}  \Bigl ( \sum_{n=j}^{\infty} n^{2j} \frac{|z_1|^{2n-2j}}{\beta _n^2}  \Bigr ) ^{1/2} 
 \le \| f || _{\beta }  \Bigl ( \sum_{n=j}^{\infty} n^{2j}\frac{|z_1|^{2n-2j}}{\beta _n^2}  \Bigr ) ^{1/2} \]
Since $\lim_{n \to \infty } (n^{2j} |z_1|^{2n}/\beta _n^2)^{1/n} = |z_1|^2/r^2 < 1$,
the root test shows that the derivative evaluation functionals are uniformly bounded on compact subsets of $\disk _r$.
Since $ \sum_{n=0}^{\infty} a_nz^n$ converges uniformly on compact subsets of $\disk _r$, $f(z)$ is analytic on $\disk _r$.

Suppose that $R > r$ and $g(z)= \sum_{n=0}^{\infty} c_nz^n$ is analytic on $\disk _R$.  For $r < R_1 < R$ this series converges absolutely for $|z| \le R_1$,
so $|c_n| < R_1^{-n}$ for sufficiently large $n$, and $\sum |c_n|^2\beta _n^2 < \infty $ by the root test.
Finally, the function $\sum_{n=1}^{\infty} z^n/(n \beta _n)$ is in $\Hardy _{\beta }$, but does not extend analytically to $\disk _R$.

\end{proof}

\subsection{Adjoint differential operators}

Assume that the weight sequence $\beta $ satisfies \eqref{multbnd} and \eqref{betabnd}, while $p_k(z) \in \Hardy _{\beta }$ for $k=0,\dots ,N$.
A differential expression $L = \sum_{k=0}^N p_k(z)D^k$, also known as a formal differential operator,  is said to have order at most $N$.
$L$ will have order equal to $N$ if $p_N(z)$ is not the zero function.  Since $Lf \in \Hardy _{\beta }$ for every $f \in \dommin $, 
the minimal operator $\dopmin :\dommin \to \Hardy  _{\beta }$ acting by $\dopmin f = Lf$ is then a densely defined operator on $\Hardy _{\beta }$.   
$\dopmin$ has a Hilbert space adjoint operator $\dopmin ^*$, but expectations about the adjoint based on traditional
integration by parts computations may be misleading.  Typically, $\dopmin ^*$ will not be a differential operator.

The concept of an adjoint pair of operators \cite[p. 167]{Kato} is useful for this discussion.
Two differential operators $\dopmin$ and $\dopmin ^+$ are adjoint to each other if
\[\langle \dopmin  f,g \rangle _{\beta } = \langle f,\dopmin ^+g \rangle _{\beta }\]
for all $f, g \in \dommin$.  $\dopmin ^*$ will then be an extension of $\dopmin ^+$. 
The corresponding expressions $L$ and $L^+$ will be called formal adjoints.
$L$ is formally symmetric if $L = L^+$.
Being part of an adjoint pair of differential expressions on a space $\Hardy _{\beta }$ is rather restrictive.
The next result is similar to one in \cite{Villone1}, where it is assumed that $L = L^+$.

\begin{thm} \label{adjform}
Suppose $p_k(z)$ and $q_k(z)$ are in $\Hardy _{\beta }$ for $k=0,\dots ,N$, with $\beta $ satisfying \eqref{multbnd} and \eqref{betabnd}.
If $L = \sum_{k=0}^N p_k(z)D^k$ has a formal adjoint $L^+ = \sum_{k=0}^N q_k(z)D^k$,
then $L $ has polynomial coefficients, with $\deg (p_k(z)) \le N+k$.  
\end{thm}

\begin{proof}

With respect to the basis $\eqref{orthobasis}$, the matrix elements of the differential expressions satisfy
\[\langle L e_m, e_n \rangle _{\beta } = \langle e_m, L^+ e_n \rangle _{\beta } = \overline{\langle L^+ e_n, e_m \rangle _{\beta }}, \quad m,n = 0,1,2,\dots .\]
Differentiation reduces the degree of $e_m$, while multiplication of $e_m$ by an analytic function $p_k(z)$ produces terms of equal or higher degree.
Since $L$ has order at most $N$,
\[\langle L e_m, e_n \rangle _{\beta } = 0, \quad n < m - N.\] 
$L^+$ has similar behaviour, so
\[ \overline{\langle L e_m, e_n \rangle _{\beta }} = \langle L ^+  e_n, e_m \rangle _{\beta } = 0, \quad  m < n - N,\]
or 
\begin{equation} \label{melbnd}
\langle L e_m, e_n \rangle _{\beta } = 0, \quad N < |m-n|  .
\end{equation} 

The proof now proceeds by induction on the coefficient index, starting with $k=0$.
Write $p_0(z) = \sum_{j=0}^{\infty } c_jz^j $ and notice that
\[\langle L e_0, e_n \rangle _{\beta } = \beta _0^{-1} \langle p_0(z),e_n \rangle _{\beta } .\]
The bound \eqref{melbnd} means that $c_j = 0$ for $j > N$, so $p_0(z)$ is a polynomial of degree at most $N$.
Suppose $p_k(z)$ is a polynomial of degree at most $N+k$ for $k < M$.  Letting $p_M(z) = \sum_{j=0}^{\infty } c_jz^j $,  
and using $D^{K}e_M = 0$ for $K >M$, 
\[\langle L e_M, e_n \rangle _{\beta } = \langle p_M(z)D^M e_M, e_n \rangle _{\beta } + \sum_{k=0}^{M-1} \langle p_k(z)D^ke_M,e_n\rangle _{\beta },\]
with
\[\langle p_M(z)D^M e_M, e_n \rangle _{\beta } = \beta _M^{-1} \langle p_M(z)D^M z^M, e_n \rangle _{\beta } = \beta _M^{-1} M! \langle p_M(z), e_n \rangle _{\beta }.\]
The induction hypothesis means that 
\[\sum_{k=0}^{M-1} \langle p_k(z)D^ke_M,e_n\rangle _{\beta } = 0 , \quad n > M+N,\]
while $\langle L e_M, e_n \rangle _{\beta } = 0$ for $n > M + N$ by  
\eqref{melbnd}, so $c_j = 0$ for $j > M+N$, and $p_M(z)$ is a polynomial of degree at most $M+N$.
\end{proof}

\begin{prop} \label{adjring}
With the hypotheses of \propref{adjform}, for each $L$ the formal adjoint $L^+$ is unique.
The set of expressions $L$ with $\Hardy _{\beta }$ differential expression adjoints is a complex algebra $\Dalg $ closed under the adjoint operation. 
\end{prop}
  
\begin{proof}
Suppose 
\[\langle Le_m, e_n \rangle _{\beta } = \langle e_m,L_1^+ e_n \rangle _{\beta } = \langle e_m,L_2^+ e_n \rangle _{\beta }, \quad m,n = 0,1,2,\dots , \]
with $L_1^+ - L_2^+ = \sum_{k=0}^N q_k(z)D^k$.  Starting with $k=0$, an induction shows that 
\[0 = \langle e_m, (L_1^+ - L_2^+)e_k \rangle _{\beta } =  \beta _k^{-1} k! \langle e_m, q_k(z) \rangle _{\beta } , \quad m = 0,1,2,\dots ,\] 
for each $k$, so $L_1^+ = L_2^+$.

The formulas 
\[(L_1+L_2)^+ = L_1^+ + L_2^+, \quad (L_1L_2)^+ = L_2^+L_1^+, \quad (cL)^+ = {\overline c}L^+, c \in \complex \]
are then easily verified.

\end{proof}
  
\subsection{First order expressions and weight restrictions}

It is easy to check that the expressions $L = zD$, and more generally polynomials in $L$, have formal adjoints. 
Since $Lz^n = nz^n$, the space $\Hardy _{\beta }$ has an orthonormal basis of eigenfunctions for $\dopmin$, 
which is thus essentially self adjoint.  For $a_1,b_0 \in \complex$, if $L = a_1z D + b_0$, then $L^+ =  \overline{a_1}z D + \overline{b_0}$
is adjoint to $L$ on $\Hardy _{\beta }$ without restrictions on $\beta$.  This example is unusual.
For a more general expression of order $1$ satisfying the degree conditions of \thmref{adjform}, the existence of a formal adjoint
of the same order imposes additional constraints on both the weights and the coefficients of the expression.  
The weight restrictions are quite severe.

 Consider finding first order adjoint pairs,
\[L_1  = p_1(z)D+ p_0(z), \quad L_1^+ = q_1(z)D + q_0(z).\]  
By \thmref{adjform}, the expressions $L_1$ and $L_1^+$ have the form  
\begin{equation}  \label{foform}
L_1  = [a_2z^2 + a_1z + a_0]D + [b_1z +b_0], \quad 
L_1^+ = [c_2z^2 + c_1z + c_0]D + [d_1z +d_0].
\end{equation}
The coefficients are constrained by the requirement $ \langle L_1 z^m, z^n \rangle _{\beta } = \langle z^m, L_1^+ z^n \rangle _{\beta } $ for $m,n = 0,1,2,\dots $,
or 
\begin{equation} \label{case1}
\langle m[a_2z^{m+1} + a_1z^m +a_0z^{m-1} ] + b_1z^{m+1} + b_0z^m, z^n \rangle _{\beta }
\end{equation}
\[ = \langle z^m, n[c_2z^{n+1} + c_1z^n +c_0z^{n-1} ] + d_1z^{n+1} + d_0z^n \rangle _{\beta } .\]

The inner products are zero if $|m-n| > 1$.  For the cases $m=0$, $n=0,1$ and $n=0$ , $m=0,1$ the equations \eqref{case1} are equivalent to  
\begin{equation} \label{ipcase0}
\beta _0^2b_0 = \beta _0^2\overline{d_0}, \quad \beta _1^2b_1 = \beta_0^2\overline{c_0}, \quad \beta _0^2a_0 = \beta _1^2\overline{d_1} .
\end{equation}
For $m \ge 1$ and $n \ge 1$, \eqref{case1} is equivalent to the equations
\begin{equation} \label{gencase}
n= m+1: \beta _{m+1}^{2}(ma_2+b_1) =  \beta _m^{2}(m+1)\overline{c_0} ,
\end{equation}
\[ n = m: ma_1 + b_0 = m\overline{c_1} + \overline{d_0},\] 
\[n = m-1: \beta _{m-1}^{2} ma_0  = \beta _{m}^{2}[(m-1)\overline{c_2}+ \overline{d_1}] .\]
The first and third equations in \eqref{gencase} can be recast in the same form,
\begin{equation} \label{recform}
 \beta _{m+1}^{2}(ma_2+b_1) =  \beta _m^{2}(m+1)\overline{c_0} , \quad
\beta _{m+1}^{2}(m\overline{c_2}+ \overline{d_1}) = \beta _{m}^{2} (m+1)a_0 .
 \end{equation}
Eliminating the weights leaves
\begin{equation} \label{ocase}
(m\overline{c_2}+ \overline{d_1})\overline{c_0} = (ma_2+b_1)a_0 .
\end{equation}

The equations \eqref{ipcase0}, \eqref{gencase}, and \eqref{ocase} provide the following relations.
\begin{equation}  \label{relations}
 \overline{d_0} = b_0, \quad   \beta _0^2 \overline{c_0} =  \beta _1^2 b_1, \quad \beta _1^2 \overline{d_1} = \beta_0^2 a_0,
\end{equation}
\[\overline{c_1} = a_1,\quad \overline{c_0}\overline{d_1} = a_0b_1,\quad  \overline{c_0}\overline{c_2} = a_0a_2.\]

\begin{thm} \label{siga2}
Let $\{ \beta _n \}$ be a positive sequence.
Suppose $L_1  = [a_2z^2 + a_0]D + b_1z $  
has a formal adjoint $L_1^+ = [c_2z^2 + c_1z + c_0]D + [d_1z +d_0]$
on $\Hardy _{\beta }$.  If $a_2 \not= 0$, then $b_1 = \sigma a_2$ for some $\sigma > 0$.
The weights satisfy  
\begin{equation} \label{wtdef0}
\beta _{m+1}^{2}= \beta _{m}^{2}\frac{\beta _1^2}{\beta _0^2}\frac{(m+1)\sigma }{m+\sigma },
\end{equation}
and $\Hardy _{\beta }$ is a Hilbert space of functions analytic on $\disk _r$ with $r^2 = \beta _1^2 \sigma /\beta _0^2$.

\end{thm}

\begin{proof}
If $a_2 \not= 0$, then \eqref{recform} implies $c_0 \not= 0$, and then $b_1 = \beta _0^2 \overline{c_0}/\beta _1^2 \not= 0$ from \eqref{relations}.
Using $\overline{c_0} = \beta _1^2 b_1/\beta _0^2$ in \eqref{recform} gives
\begin{equation} \label{betas}
\beta _{m+1}^{2} =  \beta _m^{2}\frac{\beta _1^2}{\beta _0^2}\frac{(m+1) b_1}{ma_2 + b_1}.
\end{equation}
The weights $\beta _m $ are positive, so 
\[ \lim_{m \to \infty}\frac{\beta _{m+1}^{2} }{\beta _m^{2}} = \frac{\beta _1^2}{\beta _0^2}\frac{b_1}{a_2 } > 0, \]
and $b_1 = \sigma a_2$ for some $\sigma > 0$.

The weight sequence now satisfies \eqref{wtdef0} so by induction
\[\beta _m^2 = \beta _1^2(\frac{\beta _1^2 \sigma }{\beta _0^2})^{m-1}\prod_{k=1}^{m-1} \frac{k+1}{k + \sigma }=  \beta _1^2(\frac{\beta _1^2 \sigma }{\beta _0^2})^{m-1}\prod_{k=1}^{m-1} \frac{1+1/k}{1 + \sigma/k }, \quad m \ge 2  .\]
A Taylor expansion gives
\[\lim_{m \to \infty} \Bigl [ \prod_{k=1}^{m-1} \frac{1+1/k}{1 + \sigma/k } \Bigr ]^{1/m} = \lim_{m \to \infty}\exp \Bigl (\frac{1}{m} \sum_{k=1}^{m-1} \log(1+ 1/k) - \log(1 + \sigma /k ) \Bigr ) \]
\[= \lim_{m \to \infty} \exp \Bigl (\frac{1}{m}\sum_{k=1}^{m-1} [ \frac{1-\sigma }{k} + O(k^{-2})]  \Bigr ) = 1,\]
so
\[ \lim_{m \to \infty} \beta _m^{2/m} = \frac{\beta _1^2 \sigma }{\beta _0^2}.\]
As in \propref{bndpe}, the natural domain for functions in $\Hardy _{\beta }$ is $\disk _r$.
\end{proof}

Define a collection of restricted weight sequences $\beta $ with 
\begin{equation} \label{wtdef}
\sigma > 0, \quad \beta _0 =1, \quad \beta _1^2 = 1/\sigma , \quad  
\beta _{m+1}^{2}= \beta _{m}^{2}\frac{m+1 }{m+\sigma }, \quad m \ge 1.
\end{equation}
The proof of \thmref{siga2} and \propref{bndpe} show that $\disk $ is the natural domain for $\Hardy _{\beta }$.
With these weight restrictions the adjoint pairs $L_1, L_1^+$ on $\Hardy _{\beta }$ have a simple description.  

\begin{thm} \label{char2}
Let $\beta $ be a weight sequence satisfying \eqref{wtdef}.  On $\Hardy _{\beta }$ every differential expression
\[L_1  = [a_2z^2 + a_1z +  a_0]D + \sigma a_2z + b_0 \]  
has a formal adjoint  
\[L_1^+  = [ \overline{a_0}z^2 + \overline{a_1}z + \overline{a_2}]D +  \sigma \overline{a_0} z + \overline{b_0}  .\]

\end{thm}

\begin{proof}
Using the notation of \eqref{foform}, notice first that the equations of \eqref{ipcase0} are satisfied.
Since the terms $\beta _m$ are given by \eqref{wtdef}, the equations of \eqref{gencase} become
\[n= m+1: (m+1)(m+\sigma )a_2 =  (m+\sigma )(m+1)\sigma a_2/\sigma ,\]
\[ n = m: ma_1 + b_0 = ma_1 + b_0,\] 
\[n = m-1: (m-1+\sigma )ma_0  = m[(m-1)+ \sigma ]a_0 ,\]
so the equations of \eqref{gencase} are satisfied.
Since \eqref{ipcase0} and \eqref{gencase} are equivalent to the satisfaction of \eqref{case1},
$L_1$ and $L_1^+$ are an adjoint pair.
\end{proof}

The leading coefficient $p_1(z)$ of $L_1$ may be an arbitrary polynomial of degree at most $2$. 
The roots of the leading coefficients of  $L_1$ and $L_1^+$ are related by a simple transformation.

\begin{cor} \label{roots1}
Suppose $\beta $ satisfies \eqref{wtdef}, $p_1(z) = a_2z^2 + a_1z + a_0$ is the leading coefficient of $L_1$,
and $L_1^+$ is the formal adjoint of $L_1$ on $\Hardy _{\beta }$, with leading coefficient $q_1(z)$.  If $z_1 \not= 0$, then $p_1(z_1) = 0$ if and only if
$q_1(1/\overline{z_1}) = 0$.  
\end{cor} 

\begin{proof}
By \thmref{char2} the leading coefficient of $L_1^+$ is $q_1(z) = \overline{a_0}z^2 + \overline{a_1}z + \overline{a_2}$.
If $p_1(z_1) = 0$ then
\[0 = \overline{p_1(z_1)} = \overline{z_1^2}[ \overline{a_2} + \overline{a_1}/\overline{z_1} + \overline{a_0}/\overline{z_1}^{2}]
= \overline{z_1^2}q_1(1/\overline{z_1} ).\]
\end{proof}
 
\subsection{Algebraic properties of adjoint pairs}
 
Henceforth, the weight sequence $\beta $ is assumed to satisfy \eqref{wtdef}.  
In the proof of \thmref{siga2} it was noted that \eqref{betabnd} holds with $r=1$.
It is easy to check that \eqref{multbnd}, and so \propref{bndop}, also hold.
Let $\Dalg $ denote the complex algebra
of differential expressions $L$ with a formal $\Hardy _{\beta } $ adjoint expression $L^+$.

\begin{thm} \label{adjchar}  
As an algebra, $\Dalg$ is generated by its expressions with order at most one.  
If $L = \sum_{k=0}^N p_k(z)D^k \in \Dalg $, with $p_N(z) = \sum_{j=0}^{2N} c_jz^j$, then
the leading coefficient of $L^+$ is  $q_N(z) = \sum_{j=0}^{2N} \overline{c_j} z^{2N-j}$.
If $z_1 \not= 0$, then $p_N(z_1) = 0$ if and only if $q_N(1/\overline{z_1}) = 0$.  
\end{thm}

\begin{proof}
The claim that $\Dalg$ is generated by its expressions with order at most one is proved by induction on the order, with the case of order at most $1$ trivially valid.
Suppose the result is true for order less than $N$ and assume $L$ has order $N$ with (nonzero) leading coefficient $p_N(z)$. 
By \thmref{adjform}, $p_N(z)$ is a polynomial of degree $K$, with $K \le 2N$, which may be written in factored form
\[p_N(z)  =  \alpha (z - z_1)\dots (z-z_K) .\]

By \thmref{char2}, the expressions in $\Dalg $ with order at most one may have any polynomial leading coefficient
of degree at most $2$.  Recall that if ${\cal F}_1, \dots , {\cal F}_N$ are differential expressions and the product
$P = {\cal F}_1 \cdots {\cal F}_N$ is written in the standard form \eqref{dexpress}, then the leading coefficient of $P$ is the product of the
leading coefficients of the ${\cal F}_m$.  Since $\deg (p_N(z)) \le 2N$ there is a product $P$ of $N$ first order expressions ${\cal F}_1, \dots , {\cal F}_N  \in \Dalg$
whose leading coefficient matches that of $L$.  $P$ has the formal adjoint ${\cal F} _N^+ \dots {\cal F} _1^+$.
The difference $L - P$ has strictly lower order than $N$.
By the induction hypothesis $L- P$ is in the algebra generated by expressions of order at most one,   
and so is $L $.

Suppose that for $m = 1,\dots ,N$ the expressions ${\cal F} _m$ have leading coefficients $a_2(m)z^2 + a_1(m)z + a_0(m)$.
Then
\[p_N(z) = \sum_{j=0}^{2N} c_jz^j = \prod_{m=1}^N (a_2(m)z^2 + a_1(m)z + a_0(m)) ,\] 
and if $q_N(z)$ denotes the leading coefficient of $L^+$, then by \thmref{char2}
\[q_N(z) = \prod_{m=1}^N (\overline{a_0}(m)z^2 + \overline{a_1}(m)z + \overline{a_2}(m))\]
\[ = z^{2N} \prod_{m=1}^N (\overline{a_0}(m) + \overline{a_1}(m)z^{-1} + \overline{a_2}(m)z^{-2})
= z^{2N}  \sum_{j=0}^{2N} \overline{c_j} z^{-j} = \sum_{j=0}^{2N} \overline{c_j} z^{2N-j}.\]

Finally, if $z_1 \not= 0$ and $p_N(z_1) = 0$, then
\[0 =  \sum_{j=0}^{2N} \overline{c_j} \overline{z_1}^j = \overline{z_1}^{2N}  \sum_{j=0}^{2N} \overline{c_j} \overline{z_1}^{j-2N} =  \overline{z_1}^{2N} q_N(1/\overline{z_1} ). \]
\end{proof}

Notice that the mapping $z_1 \to 1/\overline{z_1}$ taking roots of $p_N(z_1) $ to roots of $q_N$ extends to $z_1 = 0$ and $z_1 = \infty $ in the sense that 
if $p_N(z)$ has degree $K \le 2N$ with $z=0$ a root of order $m$, then $q_N(z)$ has degree $2N - m$ with $z=0$ a root of order $2N-K$.
 
Of course $p_N(z) = q_N(z)$ when $L$ is formally symmetric, so \thmref{adjchar} provides the following corollary.

\begin{cor} \label{symroots}
Suppose $L $ is formally symmetric with order $N$.  Then the leading coefficient $p_N(z)$ has the form
\[p_N(z) = c_Nz^N + \sum_{j=0}^{N-1} [c_jz^j +  \overline{c_j}z^{2N-j}] , \quad c_N = \overline{c_N}.\]
The nonzero roots $z_j$ of $p_N(z)$ are closed under the map $z_j \to 1/\overline{z_j}$.
\end{cor}

Both \cite{Villone1} and \cite{StorkB} consider the problem of characterizing formally symmetric expressions $L$ on $\Berg$.
Villone \cite{Villone1} succeeds with a fairly complex recursive technique; a similar method appears again in \cite{Villone5}.  
Stork \cite{StorkB} shows that for $c \not= 0$ and $n = 1,2,3,\dots $ the examples
\[l_{n,r} = (cz^{n+r} + \overline{c}z^{n-r})D^n + \sum_{k=1}^r c \binom{r}{k}\frac{(n+1)!}{(n+1-k)!} z^{n+r-k}D^{n-k}, \quad  0 \le r \le n ,\]
are formally symmetric in $\Berg$ without providing a characterization of formal symmetry.
By taking advantage of simple adjoint formulas, an explicit characterization for $\Hardy _{\beta }$ is possible.

For $n = 0,1,2,\dots $ and $0 \le r \le n $ define expressions
\[B_{n,r} = (zD)^{n-r}D^{r}.\]
This is simply a product of first order expressions.  The expression $zD$ is formally symmetric, 
while the $\Hardy_{\beta }$ formal adjoint of $D$ is $z^2D+ \sigma z$ by \thmref{char2}.  Taking the adjoint factors in reverse order gives
\[B_{n,r}^+ = (z^2D + \sigma z)^{r}(zD)^{n-r}.\]
If $n=0$, then $B_{0,0}$ is simply multiplication by $1$.
Formally symmetric expressions can be obtained by adding an expression in $\Dalg $ to its adjoint expression, leading to the following observation.

\begin{lem} \label{bsym}
For $n=0,1,2,\dots $, $0 \le r \le n$, and $c_{n,r} \in \complex$, the expressions
$c_{n,r}B_{n,r} + \overline{c_{n,r}}B_{n,r}^+ $ are formally symmetric, with highest order term 
$(c_{n,r}z^{n-r} + \overline{c_{n,r}}z^{n+r})D^n$ when written in the standard form \eqref{dexpress}.  
\end{lem}

\begin{thm} \label{symmchar}
An differential expression $L$ is formally symmetric in $\Hardy _{\beta}$ if and only if it can be written in the form
\begin{equation} \label{symform}
L = \sum_{n=0}^N \sum_{r=0}^n [c_{n,r}B_{n,r} + \overline{c_{n,r}}B_{n,r}^+].
\end{equation}
\end{thm}

\begin{proof}
The given form is the sum of formally symmetric operators, so is formally symmetric.

Suppose $L$ is formally symmetric and of order $N$.  
If $N=0$ then the expression is multiplication by a real constant.
Proceeding by induction, assume a formally symmetric expression with order less than $N$
has the given form.  By \corref{symroots} and \lemref{bsym} the leading coefficient of $\dop $ can be matched
by a formally symmetric expression $\sum_{r=0}^N [c_{N,r}B_{N,r} + \overline{c_{N,r}}B_{N,r}^+] $.  
Since $L - \sum_{r=0}^N [c_{N,r}B_{N,r} + \overline{c_{N,r}}B_{N,r}^+]$ has order less than $N$,
it has the desired form by the induction hypothesis, and so $L$ has the prescribed form.
\end{proof}

\section{$\Dreg$ operators }

Although the class of expressions $L \in \Dalg$ is rather restricted, there is an interesting overlap with equations of the Fuchsian class.
When $L \in \Dalg $ has order two, eigenvalue equations $Ly = \lambda y$ include many of the classical Riemann and Heun equations,
as well as problems with five regular singular points.  As demonstrated below, the root symmetry in the leading coefficients of symmetric expressions
can lead to existence theorems for analytic eigenfunctions, a manifestation of global features of the monodromy data at regular 
singular points.

Recall that the weight sequences $\beta $ satisfy \eqref{wtdef}.
If the differential expression $L$ as in \eqref{dexpress} has coefficients $p_k(z) \in \Hardy _{\beta }$,   
the minimal operator $\dop _{min}$ may be extended to a maximal operator $\dop _{max}$ with the domain,
\begin{equation} \label{domdef}
\dom _{max} = \{ f \in \Hardy _{\beta } \ | \ Lf \in \Hardy _{\beta} \}.
\end{equation}

\begin{prop}
The operator $\dop _{max}$ with domain $\dom _{max}$ is closed in $\Hardy _{\beta }$.
\end{prop}

\begin{proof}
The argument follows \cite{Villone1}. 
Suppose $f_j \in \dom _{max}$ for  $j = 1,2,3,\dots $, and that $\{ f_j \}$ and $\{ \dop f_j \}$ are Cauchy sequences in $\Hardy _{\beta}$,
converging respectively to $f$ and $g$.  By \propref{bndpe} the sequences $\{ f_j^{(k)} \}$ converge uniformly to $f^{(k)}$ on any compact subset $K$ of $\disk$,
so $\{ \dop f_j \} $ converges uniformly on $K$, with $\dop f = g$.  
\end{proof}

With additional hypotheses, $L$ may be interpreted as a more conventional operator on $\real$. 
If the coefficients $p_k(z)$ are analytic on the closed disk $\overline{\disk }$, 
the change of variables $z = e^{it}$ on the unit circle leads to a $2\pi $-periodic expression 
\[L_t = \sum_{k=0}^N p_k(e^{it}) (-i e^{-it} \frac{d}{dt})^k , \quad t \in \real .\]
If the leading coeffiicient $p_N(z)$ has no zeros when $|z| =1$, then the leading coefficient
$p_N(e^{it})(-ie^{-it})^N$ will have no zeros for $t \in \real $.  The expression $L_t$
will then fall into the periodic``regular case" \cite[p. 188-194]{CL} \cite[p. 1280]{DS}.
Features such as closed range and compact resolvent are common for operators acting by $L_t$   
on a variety of domains consisting of $2\pi $-periodic functions, 

It seems reasonable to expect a class of such ``regular" problems on $\Hardy _{\beta }$.  
Say that $L$ is $\Dreg$ if the coefficients $p_k(z)$ are analytic on the closed unit disk $\overline{\disk }$, and $p_N(z) \not= 0$ when $|z| = 1$;
in particular the set $\{ z_k  \ | \ |z_k| < 1, p_N(z_k) = 0 \} $ is finite. 
Stork \cite{StorkB} has shown that if $L$ is $\Dreg$ and formally symmetric on $\Berg $, then the closure of the minimal operator $\dop _{min}$ is self adjoint,
with compact resolvent.  The main goal of this section is to develop the techniques needed to extend Stork's result to formally symmetric 
expressions on $\Hardy _{\beta}$, with related results for more general $\Dreg $ expressions.

\subsection{$\Hardy _{\beta }$ Sobolev spaces}

For weight sequences $\beta$ satisfy \eqref{wtdef}, and for $k=0,1,2, \dots $, 
define additional Hilbert spaces $\Hardy _{\beta }^k$ of functions $f \in \Hardy _{\beta }$ whose $k$-th derivative is also in $\Hardy _{\beta }$.
For $f = \sum_{n=0}^{\infty } c_nz^n$, and $g = \sum_{n=0}^{\infty } b_nz^n$, the $\Hardy _{\beta}^k$ inner product is
\[\langle f,g\rangle _k = \sum_{n=0}^{\infty} (1 + n^{2k})c_n\overline{b_n} \beta _n^2 ,\]
the norm is given by $\| f \| _k^2 = \langle f,f \rangle _k$, and the set of elements is
\begin{equation} \label{Sobdef}
\Hardy _{\beta }^k = \{ f = \sum_{n=0}^{\infty } c_nz^n \ | \ \sum_{n=0}^{\infty} (1 + n^{2k})|c_n|^2 \beta _n^2 < \infty \} .
\end{equation}

The fact that multiplication by $z$ acts as a bounded operator can be extended to a broader class of multiplication operators 
on the spaces $\Hardy _{\beta }^k$.  Let $C_{\beta} =  \sup_n \beta _{n+1}/\beta _n$.  The conditions \eqref{wtdef} imply $1 \le C_{\beta } <\infty $, and $C_{\beta} = 1$ for $\sigma \ge 1$.

\begin{prop} \label{bndmult}
Suppose $\phi (z) = \sum_{n=0}^{\infty} \alpha _nz^n $, with $\sum_{n=0}^{\infty} |\alpha _n|2^{nk}C_{\beta }^n < \infty $.
Then the operator $M_{\phi}f = \phi(z)f(z)$ is bounded on $\Hardy _{\beta }^k$.
\end{prop}

\begin{proof}
First consider  multiplication by $z$ on $\Hardy _{\beta }^k$.
With $ f = \sum_{n=0}^{\infty} c_nz^{n}$, and the usual operator norm $ \|M_{z} \| = \sup_{\| f \| _k \le 1} \|zf\| _k$,  
\[ \| zf \|_k^2 =  \sum_{n=0}^{\infty} (1 + n^{2k})|c_{n}|^2 \beta_{n}^2 \frac{\beta _{n+1}^2}{\beta _{n}^2} \frac{(1+ [n+1]^{2k})}{(1+n^{2k})} \]
\[ \le \sup_n \frac{\beta _{n+1}^2}{\beta _{n}^2} \frac{1+ [n+1]^{2k}}{1+n^{2k}} \| f \|_k^2 \le 2^{2k}C_{\beta }^2\| f \| _k^2 ,\]
so $M_z$ is a bounded operator on $\Hardy _{\beta }^k$ with norm bounded by $2^kC_{\beta }$. 

The operator norm is submultiplicative, and absolutely convergent series in the Banach space of bounded operators are convergent,
so $M_{\phi }$ is bounded on $\Hardy _{\beta }^k$.
\end{proof}

The next lemma uses a standard Fourier series argument.

\begin{lem} \label{Sobest}
Suppose $f \in \Hardy _{\beta }^k$ and $0 \le j < k$.  Then for any $\epsilon > 0$ there is a constant $C$, independent of $f$, such that
\begin{equation} \label{epsest}
\| f \| _j \le \epsilon \| f \| _k + C \| f \|_{\beta } .
\end{equation}
The set $B = \{ f \in \Hardy _{\beta }^k \ | \ \| f \| _k \le 1 \} $ has compact closure in $\Hardy _{\beta }^j$.

\end{lem}

\begin{proof} For $\epsilon > 0$  the inequality $ (1+n^{2j}) \le \epsilon ^2 (1+n^{2k}) + C^2 $ holds 
for $n = 0,1,2,\dots ,$ and $C$ sufficiently large.  Thus
\[ \| f \| _j^2 \le \epsilon ^2 \| f \| _k^2 + C^2 \| f \|_{\beta } ^2 \le (\epsilon \| f \| _k + C \| f \|_{\beta } )^2,\]
establishing \eqref{epsest}.

Suppose $\{ f_l(z) = \sum_{n=0}^{\infty} c_n(l)z^n \}$ is a sequence in $B$.  For each fixed $n$, the sequence $c_n(l)$ 
is bounded, so by the usual diagonalization argument \cite[p. 167]{Royden} the sequence $\{f_l(z) \}$ has a subsequence $\{ f_m(z) \}$ such that
$\{ c_n(m), m =1,2,3,\dots  \}$ is a Cauchy sequence in $\complex $ for each $n$.

Since $\| f_m(z) \|_k \le 1$, for any $\epsilon > 0$ there is an $N$ such that 
\[ \sum_{n=N}^{\infty} (1 + n^{2j})|c_n(m)|^2 \beta _n^2 < \epsilon , \quad m = 1,2,3,\dots .\]
Since the sequences $c_n(m) \in \complex $ are convergent for $n < N$, the sequence $\{f_m(z) \} $ is a Cauchy sequence in $\Hardy _{\beta}^j$.

\end{proof}

\begin{lem} \label{donly}
If $f(z)$ is analytic in $\disk $ and $D^k f \in \Hardy _{\beta }$,
then $f \in \Hardy _{\beta }^k$.  
\end{lem}

\begin{proof}
Since $f(z)$ is analytic in $\disk $, both $f$ and $D^kf$ have power series which converge for every $z \in \disk $,
\[f = \sum_{n=0}^{\infty} c_n z^n, \quad  D^kf = \sum_{n=k}^{\infty} n(n-1) \cdots (n-k+1) c_n z^{n-k} .\]
The assumption that $D^kf \in \Hardy _{\beta }$ means that
\[ \sum_{n=k}^{\infty} |n(n-1) \cdots (n-k+1) c_n \beta _{n-k}|^2 < \infty .\] 

Now $\beta _{n-k} = \beta _n \prod_{j=1}^k \beta _{n-j}/\beta _{n-j+1}$, and $\lim_{n \to \infty } \beta _{n-j}/\beta _{n-j+1}= 1$,
so $f \in \Hardy _{\beta }^k$ since
\[ \sum_{n=k}^{\infty} n^{2k} |c_n \beta _{n}|^2 < \infty .\] 

\end{proof}

\begin{lem} \label{switch}
Suppose $f(z) \in \Hardy_{\beta }^k$, $r(z) = (z-z_1) \dots (z-z_M)$, and $j$ is an integer, with $1 \le j \le k$. 
Then there is a $g \in \Hardy _{\beta}^k$ such that $r(z)D^jf = D^jg$, with $\| g \| _k \le C\| f \| _k $. 
\end{lem}

\begin{proof}
Since multiplication by $z$ is a bounded operator on $\Hardy _{\beta }^k$, the function 
\[F(z) = \sum_{n=0}^{\infty} \frac{c_n}{n+1}z^{n+1} = z   \sum_{n=0}^{\infty} \frac{c_n}{n+1}z^{n}\]
is in $\Hardy _{\beta }^k$, with $F'(z) = f(z)$.  Begin an induction with the case $r(z) = (z-z_1)$.
By the product rule,
\[ (z-z_1)D^j f =  D^j[(z-z_1)f(z) - jF(z)] = (z-z_1)D^jf + jf^{(j-1)} - jf^{(j-1)},\]
and $ \| (z-z_1)f(z) - jF(z) \| _k \le C \| f \| _k$.

Suppose the result holds when there are fewer than $M$ factors.
Then $(z-z_1)\cdots (z-z_M) D^jf = (z-z_1)D^jg $, and the first case may be applied to finish the proof.

\end{proof}

The next result provides a kind of lower bounded for the operator of multiplication by a nonzero polynomial $r(z)$
with no roots on $\partial \disk = \{ |z| = 1\}$.

\begin{thm} \label{plowbnd}
Suppose $j $  and $k $ are nonnegative integers.
If
\[r(z) = \alpha \prod_{m=1}^M (z-z_m), \quad |z_m| \notin \partial \disk , \quad \alpha \not= 0,\]
then there are  positive constants $C_1,C_2$  such that 
\begin{equation} \label{lowbnd1}
\| D^jf \|_{k} \le C_1 \| r(z) D^jf \|_{k} + C_2 \| f \|_{\beta } , 
\end{equation}
for any  $f \in \Hardy _{\beta }^{k+j}$.

\end{thm}

\begin{proof}
With $f = \sum_{n=0}^{\infty} c_nz^n$,  
\[ \| zD^jf \| _{k}^2 =  \sum_{n=0}^{\infty} (1+n^{2k})(n(n-1) \cdots (n- j + 1))^2 |c_n|^2 \beta _{n-j+1}^2\]
while
\[ \| D^jf \| _{k}^2 =  \sum_{n=0} (1 + n^{2k})(n(n-1) \cdots (n- j + 1))^2|c_n|^2\beta _{n-j+1}^2\frac{\beta _{n-j}^2}{\beta _{n-j+1}^2} .\]
Since $\lim_{n \to \infty} \beta _{n+1}/{\beta _n} = 1$, for any $\epsilon $ with $0< \epsilon < 1$ there is an $N$ such that 
\[  (1 - \epsilon)^2 \sum_{n=N}^{\infty}(1 + n^{2k}) (n(n-1) \cdots (n- j + 1))^2|c_n|^2\beta _{n-j+1}^2  \frac{\beta _{n-j}^2}{\beta _{n-j+1}^2} \]
\[ \le \sum_{n=N}^{\infty} (1+n^{2k})(n(n-1) \cdots (n- j + 1))^2  |c_n|^2\beta _{n-j+1}^2 \]
\[  \le (1 + \epsilon)^2 \sum_{n=N}^{\infty}(1 + n^{2k}) (n(n-1) \cdots (n- j + 1))^2|c_n|^2\beta _{n-j+1}^2  \frac{\beta _{n-j}^2}{\beta _{n-j+1}^2} \]
and so a $C > 0 $, depending on $\epsilon $, such that 
\begin{equation} \label{goodest1}
 (1-\epsilon )^2\| D^jf \|_{k }  ^2 \le  \| zD^jf \|_{k}  ^2 + C \| f \|_{\beta }  ^2 \le (\| zD^jf \|_{k}  + C \| f \|_{\beta }  ) ^2 
 \end{equation}
 and 
\begin{equation} \label{goodest2}
 \| zD^jf \|_{k}  ^2  \le ((1 + \epsilon )\| D^jf \|_{k}  + C \| f \|_{\beta }  ) ^2. 
 \end{equation}
 
Suppose $r(z) = (z-z_1)$, with $|z_1| > 1$.  By \eqref{goodest2} the reverse triangle inequality gives
\[  \| (z-z_1)D^jf \|_{k}  \ge |z_1| \| D^jf \|_{k} - \| zD^j f \| _k  \]
\[ \ge  |z_1| \| D^jf \|_{k} - (1 + \epsilon ) \| D^jf \|_{k}  - C \| f \|_{\beta}  .\]
Since $|z_1| > 1$, $\epsilon $ may be chosen with $1 + \epsilon < |z_1 |$, giving 
\[  \| D^jf \|_{k } \le \frac{1}{|z_1| - (1 + \epsilon )} \Bigl ( \| (z-z_1)D^jf \|_{k}  + C \| f \|_{\beta } \Bigr ) .\]

In case $r(z) = (z-z_1)$ with $|z_1| < 1$, \eqref{goodest1} gives
\[  \| (z-z_1)D^jf \|_{k}  + C \| f \|_{\beta }  \ge \| zD^jf \|_{k}   + C \| f \|_{\beta } - |z_1| \| D^jf \|_{k}  \]
\[ \ge  (1-\epsilon ) \| D^jf \|_{k}  - |z_1| \| D^jf \|_{k}  .\]
Choose $\epsilon $ so that $1-\epsilon > |z_1 |$ to get
\[  \| D^jf \|_{k } \le \frac{1}{(1-\epsilon ) - |z_1|} \Bigl ( \| (z-z_1)D^jf \|_{k}  + C \| f \|_{\beta } \Bigr ) .\]

Using \lemref{switch}, if $r_1(z) = \alpha (z-z_1) \cdots (z-z_{M-1})$, there is a $g$ such that 
$D^j g = r_1(z) D^jf $, with $\| g \| _k  \le C \| f \| _k$.
The proof of \eqref{lowbnd1} concludes by induction on the number of factors $M$, with the first case established.
Using the first case, and the assumed validity of the result with fewer than $M$ factors,
\[ \| r(z) D^j f \| _k = \| (z-z_M)r_1(z) D^jf \| _k = \| (z-z_M)D^jg \| _k \ge C_1 \| D^j g \| _k - C_2 \| g \| _{\beta } \]
\[ = C_1 \| r_1(z) D^jf \| _k - C_3 \| f \| _{\beta } \ge C_4 \| D^j f \| _k - C_5  \| f \| _{\beta } .\]

\end{proof}

\subsection{The domain of $\dop _{max}$}

The discussion of operator domains for $\Dreg $ expressions \eqref{dexpress} begins with a lemma similar to one in \cite{StorkB}.
The parameter $\sigma $ is from \eqref{wtdef}.

\begin{lem} \label{ink1}
Suppose $L$ is $\Dreg$ with order $N$.
If  $\sigma \ge 1$, then $\Hardy _{\beta }^N$ is in the domain of the closure of $\dop _{min}$. 
This result holds for all $\sigma > 0$ if $L$ has polynomial coefficients.
\end{lem}

\begin{proof}
If $f(z) \in \Hardy _{\beta }^N$, then the $m$-th order Taylor polynomials $t_m(z)$ for $f(z)$ centered at zero converge to $f$ in the $ \Hardy _{\beta }^N$ norm.
Thus for $j = 0,1,\dots ,N$ the derivatives $D^jt_m$ converge to $D^jf$ in the $\Hardy _{\beta }$ norm.
$C_{\beta } = 1$ in \propref{bndmult} since $\sigma \ge 1$.  Also,  
the coefficients $p_k(z)$ of $L$ are analytic on the closed disk $\overline{\disk}$, so the Taylor series for the coefficients converge absolutely for $|z| \le 1$.
By \propref{bndmult} multiplication by $p_k(z)$ acts as a bounded operator on $\Hardy _{\beta }$, and thus $L t_m(z)$ converges to $Lf$,
putting $f$ in the domain of the closure of $\dop _{min}$.

In case $L$ has polynomial coefficients, multiplication by $p_k(z)$ acts as a bounded operator for all  $\sigma > 0$.

\end{proof}

As a densely defined operator on a Hilbert space, $\dop _{min}$ has an adjoint operator $\dop _{min}^*$, whose graph is 
the set of all pairs $(g_1,g_2) \in \Hardy _{\beta} \oplus \Hardy _{\beta}$ such that
\[\langle f,g_2 \rangle _{\beta } = \langle \dop _{min}f,g_1 \rangle _{\beta }\]
for all polynomials $f$.  Recall that $\Dalg$ is the set of differential expressions $L$ with a formal adjoint expression $L^+$,
and that all $L \in \Dalg$ have polynomial coefficients by \thmref{adjform}.
The next result is similar to one in \cite{Villone1}, where $\dop _{min}$ is assumed to be symmetric.

\begin{thm} \label{adjoints}
If $L \in \Dalg $, then $ \dop _{min}^* = \dop ^+_{max} $.
\end{thm}

\begin{proof}
For any $f \in \dom _{min}$, the function  $\dop _{min}f = \sum_{j=0}^m b_jz^j $ is also a polynomial.  
Suppose $g= \sum_{j = 0}^{\infty}c_jz^j$ is in the domain of $\dop _{min}^*$.
Since the powers $z^j$ are an orthogonal basis for $\Hardy _{\beta }$, if $M > m + N$ then the 
orthogonal projection $g_M(z) = \sum_{j=0}^M c_jz^j$ of $g$ onto the span of $1,\dots , z^M$
is a polynomial which satisfies
\begin{equation} \label{Hadjoint}
\langle \dop _{min} f,g \rangle = \langle \dop _{min}f,g_M \rangle = \langle f,L^+ g_M \rangle = \langle f,L^+g \rangle. 
\end{equation}
Since the polynomials are dense in $\Hardy _{\beta }$, $\dop _{min}^*g= L^+ g$, and $g$ is in the domain of $\dop _{max}^+$.
In addition, \eqref{Hadjoint} shows that any $g$  in the domain of $\dop _{max}^+$ is in the domain of $\dop _{min}^*$.

\end{proof}

\begin{thm} \label{dregopsp}
Suppose $L = \sum_{k=0}^N p_k(z) D^k $ is $\Dreg $ of order $N$ with polynomial coefficients.
Then $\dom _{max} = \Hardy _{\beta }^N$ and $\dop _{max}$ is a closed operator on $\Hardy _{\beta }$.
\end{thm}

\begin{proof}
Since multiplication by $p_k(z)$ is bounded on $\Hardy _{\beta }$, $\Hardy _{\beta }^N \subset \dom _{max}$ for $L$.

The result holds trivially if $N=0$, since $f \in \Hardy _{\beta}^0$ if $f \in \dom _{max}$. 
The proof proceeds by induction on the order $N \ge 1$.
If $N=1$, then $L = p_1(z)D + p_0(z)$.  Since $p_0(z)f \in \Hardy _{\beta}$, we have
$p_1(z)Df \in \Hardy _{\beta}$, and $f \in \Hardy _{\beta }^{1}$ by \thmref{plowbnd}.
Suppose $N \ge 2$ and the result holds for $K < N$.  
  
Using the product rule, the expression $L$ can be written in the form $L = b_0(z) + D\sum_{k=1}^N b_k(z) D^{k-1}$, with $b_N(z) = p_N(z)$.
Since $Lf - b_0(z)f =  D\sum_{k=1}^N b_k(z) D^{k-1}f \in \Hardy _{\beta }$,
\lemref{donly} implies $\sum_{k=1}^N b_k(z) D^{k-1}f \in \Hardy _{\beta }^1$.
By the induction hypothesis, $f \in \Hardy _{\beta }^{N-1}$.
Finally, since $Df \in \Hardy _{\beta }$ and 
\[ Lf - p_0(z)f = [\sum_{k=1}^N p_k(z) D^{k-1}]Df \in \Hardy_{\beta },\]
the induction hypothesis gives $Df \in \Hardy _{\beta }^{N-1}$, so $f \in \Hardy _{\beta }^N$.

By \lemref{ink1}, $\dop _{max}$ is the closure of $\dop _{min}$.
\end{proof}

A result similar to \thmref{dregopsp} holds for more general coefficients when $\sigma \ge 1$.

\begin{thm} \label{dregops}
If $\sigma \ge 1$ and $L = \sum_{k=0}^N p_k(z) D^k $ is $\Dreg $ of order $N$, then $\dom _{max} = \Hardy _{\beta }^N$. 
$\dop _{max}$ is a closed operator on $\Hardy _{\beta }$.
\end{thm}

\begin{proof}
Since the coefficients $p_k(z)$ are analytic on $\overline{\disk }$, and $C_{\beta } = 1$ if $\sigma \ge 1$, 
multiplication by $p_k(z)$ is bounded on $\Hardy _{\beta }$ and $\Hardy _{\beta }^N \subset \dom _{max}$ for $L$.

Assume $f \in \Hardy _{\beta }$ and $Lf \in \Hardy _{\beta }$.  
If $z_1, \dots ,z_K$ are the roots of $p_N(z)$ in $\disk $, listed with multiplicity, let $r(z) =  (z-z_1) \cdots (z-z_K)$.
The leading coefficient can be factored as $p_N(z) = r(z)q(z)$, where
$q(z)$ is analytic with no zeros on $\overline{\disk }$.  
Since $C_{\beta } = 1$ and $1/q(z)$ has an absolutely convergent Taylor series on $\{ |z | \le 1 \}$, multiplication by $1/q(z)$ is bounded on $\Hardy _{\beta }$
by \lemref{bndmult}.  Thus $(1/q(z))Lf \in \Hardy _{\beta }$, and $L$ can be assumed to have leading coefficient $p_N(z) = r(z)$.

The induction argument from the proof of \thmref{dregopsp} may now be applied again.

\end{proof}

Suppose $L$ is a formally symmetric $\Dreg$ expression.  The next result shows that $\dop _{max}$ is self-adjoint, 
with compact resolvent $R(\lambda ) = (\dop _{max} - \lambda I)^{-1}$.

\begin{thm} \label{regsa}
Suppose $L = \sum_{k=0}^N p_k(z)D^k$  is $\Dreg$ and formally symmetric on $\Hardy _{\beta }$, with order $N \ge 1$.
Then $\dop _{max}$ is the closure of $\dop _{min}$, and $\dop _{max}$ is self adjoint.
The resolvent $R(\lambda ): \Hardy _{\beta } \to \Hardy _{\beta }^N$ is uniformly bounded on compact subsets of the
resolvent set, and $R(\lambda ): \Hardy _{\beta } \to \Hardy _{\beta }$ is compact.
\end{thm}

\begin{proof}

The minimal operator $\dop _{min}$ is symmetric on $\Hardy _{\beta }$, with $ \dop _{min}^* = \dop _{max}$ by \thmref{adjoints}.  
Let $\widetilde{\dop _{min}}$ denote the closure of $\dop _{min}$. Then $\widetilde{\dop _{min}}$ is symmetric \cite[p. 269]{Kato} with $ \widetilde{\dop _{min}}^* = \dop _{max}$
\cite[p. 168]{Kato}.  Now $\dom _{max} = \Hardy _{\beta }^N$ by \thmref{dregopsp} and 
$\Hardy _{\beta }^N$ is in the domain of $ \widetilde{\dop _{min}}$ by \lemref{ink1}.
Thus $ \widetilde{\dop _{min}}$ is an extension of $\dop _{max}$, and $ \widetilde{\dop _{min}} = \dop _{max}=  \widetilde{\dop _{min}}^* $.

For $f \in \Hardy _{\beta }^N$, the reverse triangle inequality gives
\[ (L - \lambda I)f \ge \| p_N(z)D^N f \| - \sum_{k=0}^{N-1} \| p_k(z)D^k f \| ,\]
so  \thmref{plowbnd} implies the existence of constants $C_k$, which may be chosen uniformly for $\lambda $ in compact subsets of the
resolvent set, such that    
\[ \| D^Nf \| _{\beta } \le C_N \| p_N(z)D^N f \| _{\beta } + C_N \| f \| _{\beta }  \le C_N\| (L - \lambda I ) f \| _{\beta }  + \sum_{k=0}^{N-1} C_k \| D^k f\| _{\beta }  \]
By \lemref{Sobest} the terms $\| D^k f \| _{\beta }$ may be replaced by $\epsilon \| D^N f \| _{\beta } + C \| f \| _{\beta } $ for any $\epsilon > 0$, so 
with a new constant $C$, 
\[ \| D^Nf \|  _{\beta } \le C  \| (L - \lambda I ) f \| _{\beta } + C \| f \| _{\beta } .\]
Taking $f = R(\lambda ) g$ for $g \in \Hardy _{\beta }$,
\[ \| D^N R(\lambda )g \|  _{\beta } \le C  \| g \| _{\beta } + C \| R(\lambda )g \| _{\beta } ,\]
so the resolvent $R(\lambda )$ of $\dop _{max}$ is a bounded operator from $\Hardy _{\beta }$ to $\Hardy _{\beta}^N$.
By \lemref{Sobest} the image of the unit ball has compact closure.

\end{proof}

\subsection{Fredholm index}

Suppose $T$ is a closed operator on $\Hardy _{\beta }$, while $A$ is another operator on $\Hardy _{\beta }$
whose domain includes the domain of $T$.  Recall \cite[p. 194]{Kato} that $A$ is relatively compact with respect to $T$ if 
for every bounded sequence $\{u_n \}$ in the domain of $T$, with $\{ Tu_n \}$ also bounded, the sequence $\{ Au_n \}$ 
has a convergent subsequence.  The operator $T+A$ with the domain of $T$ will be closed.

Also recall, \cite[p. 230]{Kato} that a closed operator $T$ is Fredholm if $T$ has a finite dimensional null space and a
closed range of finite codimension.  The index of a Fredholm operator is $ind(T) = dim (null \ T) - codim (range \ T) $. 
If $A$ is relatively compact with respect to the Fredholm operator $T$, then  \cite[p. 238]{Kato} the operator $T+A$ is Fredholm, with
$ind(T+A) = ind(T)$.  

\begin{lem} \label{relcomp}
Suppose $\sigma \ge 1$, $L = p_N(z)D^N$ is $\Dreg $ of order $N \ge 1$, 
and $L_0 = \sum_{k=0}^{N-1} p_k(z) D^k $ has coefficients analytic on $\overline{\disk }$.
If $\dop _0$ acts by $L_0$ and $\dop = \dop _{max}$ acts by $L$ on $\Hardy _{\beta }^N$, then $\dop _0$ is relatively compact with respect to $\dop $.
\end{lem}

\begin{proof}
Assume that $z_1, \dots ,z_K$ are the roots of $p_N(z)$ in $\disk $, listed with multiplicity.
As in \thmref{dregops}, $p_N(z)$ may be factored as $p_N(z) = r(z)q(z)$, with  $r(z) =  (z-z_1) \cdots (z-z_K)$,
and with $q(z)$ analytic with no zeros on $\overline{\disk }$.    
Since multiplication by $q(z)$ and $1/q(z)$ are bounded on $\Hardy _{\beta }$ by \propref{bndmult},
it suffices to assume that $p_N(z) = r(z)$.

Assume $u_m \in \Hardy _{\beta }^N$, and the sequences $\{ u_m \}$ and $\{ r(z)D^N u_m \}$ are bounded.
By \thmref{plowbnd} the sequence $u_m $ is bounded in $\Hardy _{\beta }^{N}$.  
For $k < N$ the terms $D^ku_m $ thus have subsequences which are convergent in $\Hardy _{\beta }$ by \lemref{Sobest},
and the same holds for $\{ \dop _0 u_m \}$. 

\end{proof}

\begin{thm} \label{Fred}
Suppose $\sigma \ge 1$, and $L = \sum_{k=0}^N p_k(z) D^k $ is $\Dreg $ of order $N \ge 1$.
If $p_N(z)$ has $K$ roots in $\disk$, counted with multiplicity, then $\dop $ with domain $\Hardy _{\beta }^N$ is Fredholm with index 
$N-K$.
\end{thm}

\begin{proof}
By \lemref{relcomp} it suffices to prove the result when $L = p_N(z)D^N$. 
The null space of $\dop $, being the polynomials with degree at most $N-1$, has dimension $N$.

For $j = 1,\dots ,J$, let $z_j$ be the distinct roots of $p_N(z)$ in $\disk $ with multiplicities $M_j$.
Factor $p_N(z) = r(z)q(z)$ with $r(z) = (z-z_1)^{M_1}\cdots (z-z_J)^{M_J}$.
Since multiplication by $q(z)$ and $q^{-1}(z)$ are bounded operators on $\Hardy _{\beta }$,
it suffices to assume that $p_N(z) = r(z)$.

To establish that the range of $L = r(z)D^N$ is closed, suppose $\{ f_m \}$ is a sequence in $\Hardy _{\beta }^N$,
$h_m = r(z)D^Nf_m$, and the sequence $\{ h_m \}$ converges to $h$ in $\Hardy _{\beta}$.
By \lemref{switch} there is a sequence $\{ g_m \}$ in $\Hardy _{\beta }^N$ such that $h_m = D^N g_m$.
The first $N$ terms of the power series for $g$ may be discarded, 
\[g_m = \sum_{k=0}^{\infty} a_kz^k, \quad \widetilde{g_m} = \sum_{k=N}^{\infty} a_kz^k, \]
giving $h_m = D^N \widetilde{g_m}$.  The sequence $\{ \widetilde{g_m} \}$ converges in $\Hardy _{\beta }^N$,
and $h$ is in the range of $\dop _{max}$ since $\Dreg$ operators are closed by \thmref{dregops}. 

If $h$ is in the range of $\dop $, then $h$ has a zero of order at least $M_j$ at $z_j$.  Using \propref{bndpe}, the $K$ functionals given by $f^{(l)}(z_j)$ for $l = 0,\dots ,M_j-1$ 
are independent continuous linear functionals on $\Hardy _{\beta }$. 
By the Riesz representation theorem there are $K$ independent elements $h_1,\dots ,h_K$ of $\Hardy _{\beta }$
orthogonal to the range of $p(z)D^N$, which thus has codimension at least $K$.

If $r(z)$ is any polynomial, then $r(z) = p(z)s(z) + t(z)$ where $s(z)$ and $t(z)$ are polynomials and 
${\rm deg} \ t(z) < K$.  (For instance $t(z)$ could be in the span of the Lagrange basis for the roots of $p$).
Since the polynomials are dense in $\Hardy _{\beta }$, the polynomials $p(z)s(z)$ are dense in the range of $\dop $.
Thus the codimension of the range is at most $K$.
\end{proof}

The assumption that $\sigma \ge 1$ may be dropped if $p_N(z)$ is a polynomial.  
The simplified proofs of \lemref{relcomp} and \thmref{Fred} are omitted.

\begin{thm} \label{Fredp}
Suppose $L = \sum_{k=0}^N p_k(z) D^k $ is $\Dreg $ of order $N \ge 1$.
If $p_N(z)$ is a polynomial with $K$ roots in $\disk$, counted with multiplicity, then $\dop $ with domain $\Hardy _{\beta }^N$ is Fredholm with index 
$N-K$.
\end{thm}

\section{Eigenvalues}
 
Consider the first order expression $L = a_2(z-z_1)(z-z_2)D + \sigma a_2z + b_0$, where
$a_2 \not= 0$, $|z_1| < 1$, $|z_2| > 1$.  Elementary computations show that 
the eigenvalues $\lambda _n$ of the maximal operator $\dop $ are 
\[\lambda _n = b_0 + a_2[\sigma z_1 - (z_2 - z_1)n], \quad n=0,1,2,\dots ,\] 
with eigenfunctions 
\[y_n(z) = C \Bigl ( \frac{z - z_1}{z - z_2} \Bigr )^{n} (z - z_2)^{-\sigma } .\]
Except for the restriction to $n \ge 0$, $\{ \lambda _n \}$ is similar to an eigenvalue sequence 
for periodic eigenfunctions of a first order periodic expression on $\real $. 

The connection linking expressions on $\Hardy _{\beta }$ and periodic problems on $\real $ will be made explicit.
These efforts start by reinterpreting $L$ as an expression on $\real $ with periodic coeffiicients.
Introduce the Hilbert space $L^2_{per}$ of $2\pi $ periodic functions 
which are (Lebesgue) square integrable on $[0,2\pi ]$ with inner product
$\langle f,g \rangle = \int_0^{2\pi} f(\theta ) \overline{g(\theta )} d \theta $.
When $L$ is $\Dreg$ and formally symmetric, the periodic eigenvalue problem is a perturbation of an eigenvalue problem
for a self-adjoint operator on $L^2_{per}$.

\subsection{Periodic expressions on $\real $}

Rather than considering $\Hardy _{\beta }$ as a space of functions analytic on $\disk$, another interpretation in available.
Begin with the complex vector space of trigonometric polynomials $f:\real \to \complex $ having the form
$f(\theta ) = \sum_{n=0}^{\infty } c_ne^{in\theta }$, with only finitely many nonzero coefficients $c_n$.
As before, define the inner product
\[\langle f_1,f_2 \rangle _{\beta} = \sum_{n=0}^{\infty} b_n \overline{c_n} \beta _n^2,\]
for polynomials $f_1 = \sum_{n=0}^{\infty } b_n\exp(in\theta )$ and $f_2 = \sum_{n=0}^{\infty } c_n\exp (in \theta )$.
The completion of this inner product space is a Hilbert space, denoted $\Hardy _{\beta ,\real }$.
The map $\Hardy _{\beta ,\real } \to \Hardy _{\beta }$ given by
\[f(\theta ) \to f(z) = \sum_{n=0}^{\infty } c_nz^n ,\]
is an isometric bijection of Hilbert spaces. 
(Although function notation is used, elements of $\Hardy _{\beta ,\real }$ are typically periodic distributions on $\real $.)

This mapping takes the expression $iz \frac{d}{dz} $ from $\Hardy _{\beta}$ to $\frac{d}{d \theta }$ on $\Hardy _{\beta ,\real }$.
The more general expressions  $L = \sum_{k=0}^N p_k(z)D^k$ become $L_{per} = \sum_{k=0}^N p_k(e^{i\theta })(-ie^{-i\theta } \frac{d}{d\theta })^k$.
This reinterpretation of differential expressions is particularly useful for locating the eigenvalues 
of $\Dreg $ formally symmetric expressions, whose eigenfunctions become $2\pi$-periodic for $L_{per}$.  
Eigenvalue estimates can also be developed for certain nonselfadjoint operators on $\Hardy _{\beta }$.
  
\corref{symroots} showed that the highest order term of a formally symmetric expression on $\Hardy _{\beta }$ has the form
\[p_N(z)D^N = \Bigl ( a_Nz^N + \sum_{j=0}^{N-1} [a_jz^j +  \overline{a_j}z^{2N-j}]\Bigr ) D^N , \quad a_N = \overline{a_N}.\]
Using the polar form $a_j = |a_j|e^{i\phi _j}$, the corresponding expression on $\real $ is
\[ \Bigl ( a_Ne^{iN\theta }  + \sum_{j=0}^{N-1} [a_je^{ij \theta } +  \overline{a_j}e^{i(2N-j)\theta}]\Bigr) (-i e^{-i\theta }\frac{d}{d\theta })^N .\]
Moving the derivatives to the right and displaying the highest order terms gives 
\[(-i)^NP_N(\theta )\frac{d^N}{d\theta ^N} + \dots = (-i)^N\Bigl ( a_N  + \sum_{j=0}^{N-1} 2|a_j|\cos ([N-j]\theta -\phi _j )  \Bigr) \frac{d^N}{d\theta ^N}) + \dots .\] 
Note that $P_N(\theta )$  is a real-valued function.

The periodic expressions $L_{per}$ coming from formally symmetric expressions $L$ on $\Hardy _{\beta}$ are typically not
formally symmetric on $L^2_{per}$.  In the first order case, formally symmetric expressions on $\Hardy _{\beta}$ have the form 
\[L  = [a_2z^2 + a_1z +  \overline{a_2}]D + \sigma a_2z + b_0, \quad  a_1 = \overline{a_1}, \quad b_0 = \overline{b_0}. \]
With $a_2= |a_2|e^{i\phi}$, the corresponding periodic expressions are
\[L_{per} = -2i [\frac{a_1}{2} + |a_2|\cos(\theta+\phi  ) ]  \frac{d}{d \theta } + \sigma a_2e^{i\theta } + b_0 . \]
The term  $\sigma a_2e^{i\theta }$ is typically not real-valued.  Also note that the leading coefficient may have zeros.  
A simple second order example starts with
\[ L = c_1(zD)^2 +  c_2zD^2 + \overline{c_2}  (z^2D + \sigma z )zD, \quad c_1 = \overline{c_1} . \]
With $c_2 = |c_2|e^{i\phi}$, the corresponding periodic expression is
\[L_{per} = [-c_1 - 2|c_2|\cos (\theta - \phi ) ]\frac{d^2}{d \theta ^2} + i (c_2e^{-i\theta } -   \overline{c_2} \sigma e^{i\theta } ) \frac{d}{d \theta }  .\]

\subsection{$\Dreg$ expressions}

Eigenvalue estimates for perturbations of self-adjoint or normal operators often depend on estimates
for the separation of the eigenvalues of the unperturbed operator.  In anticipation of such an argument,
an elementary number theoretic result is needed.

\begin{lem} \label{numstuff}
Suppose $\tau > 0$, $C \in \complex$, $N \ge 2$ is an integer, and 
\[ \gamma _n = (n /\tau +C)^N ,  \quad n = 0,\pm 1,\pm 2,\dots .\]

There is a $K > 0$ such that for all $n$ with $|n|$ sufficiently large, 
\begin{equation}  \label{gamdif}
| \gamma _n - \gamma _m | \ge Kn^{N-1}, \quad \gamma _m  \not= \gamma _n. 
\end{equation}
If $2\tau C $ is an integer, $N$ is even, and $|n|$ is sufficiently large, then there is a unique $m \not= n$ with $\gamma _m = \gamma _n$.
If  $|n|$ is sufficiently large and either $2\tau C $ is not an integer or $N$ is odd, 
there is no $m \not= n$ with $\gamma _m = \gamma _n$.

\end{lem}

\begin{proof}

The difference $\gamma _n - \gamma _m$ can be written  
\[ \tau ^{-N}[ (n + C_1)^N - (m + C_1)^N], \quad C_1 = \tau C,\]
so it suffices to consider the differences $(n + C_1)^N - (m + C_1)^N$.
Notice that $2\tau C $ is an integer exactly when $2C_1 $ is an integer.

When $|n|$ is large, $|(n+ C_1)^N|$ is a strictly increasing function of $|n|$.
If $N$ is even and $2C_1 $ is an integer, then $\gamma _n = \gamma _m$ for $m  = -n - 2C_1$.
The monotonicity means there is a unique $m \not= n$ with $\gamma _m = \gamma _n$.
Also,  
\[ \lim_{n \to +\infty } \frac{\gamma _{n+1} - \gamma _n}{n^{N-1}} 
=  \lim_{n \to +\infty } \frac{1}{n^{N-1}} \Bigl [ (n+1)^N\frac{(n+1+C_1)^N}{(n+1)^N} - n^N\frac{(n+C_1)^N}{n^N}  \Bigr ]  \] 
\[ =  \lim_{n \to +\infty } \frac{1}{n^{N-1}} \Bigl [ Nn^{N-1} + NC_1(n+1)^{N-1} - NC_1n^{N-1}  \Bigr ] = N,  \] 
establishing \eqref{gamdif} when $N$ is even and $2C_1 $ is an integer.
Simple modifications of this argument also establish the lemma if $N$ is odd.

To handle the cases when $N$ is even and $2C_1 $ is not an integer,
assume now that $z =n+C_1$, $w = m+C_1$, and $|n|$ is large enough that $|n+C_1| \ge (|n|+1)/2$.
If in addition $\big | |z| - |w| \big | \ge 1$, then
\[ |z^N - w^N| \ge \Big | |z|^N - |w|^N \Big | = \Big | \int_{|w|}^{|z|} Nx^{N-1} \ dx \Big |  \ge N(\frac{|n|-1}{2})^{N-1} .\]
The remaining values of $m,n$ to consider must satisfy   
$ -1 <  |n+C_1| - |m+C_1|  < 1$, implying $-1-2|C_1| < |n| - |m| < 1 + 2|C_1|$.
In other words, there is a constant $K$ such that \eqref{gamdif} is satisfied except possibly for 
\[ m =  n - K , \dots ,  n + K, \quad {\rm or} \quad m = -n - K,\dots ,-n +K .\]

Use the identity
\[x^N - y^N = (x-y)\sum_{j=0}^{N-1} x^jy^{N-1-j}, \]
to get 
\[ \gamma _n - \gamma _m =  (n + C_1)^N - (m + C_1)^N
= (n-m) \sum_{j=0}^{N-1} (n+C_1)^j(m+C_1)^{N-1-j} .\]
When $m = n+k$ with $k \not= 0$,
\[ \gamma _n - \gamma _m =  k \sum_{j=0}^{N-1} (n+C_1)^j(n+k+C_1)^{N-1-j} .\]
The sum is a polynomial of degree $N-1$ in $n$ with highest order term $kNn^{N-1}$.

Next take $m=-n+k$, where 
\[ \gamma _n - \gamma _m =  (2n-k) \sum_{j=0}^{N-1} (n+C_1)^j(-n+k+C_1)^{N-1-j} .\]
Since $N$ is even, $\gamma _n - \gamma _m$ is a polynomial in $n$ of degree at most $N-1$, with the degree $N-1$ term, 
\[ 2n[\sum_{j=0}^{N-1} (-1)^{N-1-j} jC_1 n^{N-2} +  \sum_{j=0}^{N-1} (-1)^{N-2-j}  (N-1-j)(k+C_1)n^{N-2} ] \]
\[ = 2n[\sum_{j=0}^{N-1} (-1)^{N-1-j} jC_1 n^{N-2} + (k+C_1)\sum_{j=0}^{N-1} (-1)^{N-1-j}  j n^{N-2} ] \]
\[ = 2( k + 2C_1) n^{N-1}\sum_{j=0}^{N-1} (-1)^{N-1-j}  j  = 2( k + 2C_1) n^{N-1}[(N-1) - (N-2)/2]. \]

Since $k+2C_1 \not= 0$, the inequality \eqref{gamdif} holds for large enough $|n|$ in each of the cases $m = n+k$ and $m=-n+k$ for $-K \le k \le K$,
completing the proof.

\end{proof}
 
The next theorem will provide detailed information about the eigenvalues of a class of $\Dreg$ maximal operators.
Motivated by the self-adjoint case, the main hypothesis describes well-behaved polynomial leading coefficients $p_N(z)$. 
By \corref{symroots}, if $L$ is a formally symmetric expression of order $N$, the nonzero roots $z_j$ of $p_N(z)$ are closed under the map $z_j \to 1/\overline{z_j}$.
If in addition $L$  is $\Dreg$ then the maximal operator 
must have Fredholm index $0$, so by \thmref{Fred} there must be $N$ roots $z_j \in \disk $.

Given $N \ge 1$, say that a polynomial $p_N(z)$ is $\Rsym$ if $p_N(z)$ has degree at most $2N$, and the following conditions are satisfied:

(i) $p_N(z)$ has $N$ roots $z_m \in \disk $, with roots $z_1,\dots , z_M $ not equal to zero and listed with multiplicity,
and $z = 0$ a root of multiplicity $N-M$.

(ii) the remaining roots of $p_N(z)$ are  $1/\overline{z_1},\dots , 1/\overline{z_M} $.

\noindent If $p_N(z)$ is $\Rsym$ then there is a nonzero constant $C_1 \in \complex$ such that
\begin{equation} \label{pNform}
p_N(z) = C_1\rho _1(z) \cdots \rho _M(z) z^{N-M},  \quad \rho _m(z) = \overline{z_m}(z-z_m)(z - 1/\overline{z_m}) .
\end{equation}
Note that if $z_m = |z_m|e^{i\phi _m}$, and $z = e^{i\theta }$ lies on the unit circle, then 
\[\rho _m(e^{i\theta }) =   |z_m| e^{-i \phi _m}e^{2i\theta } - (|z_m|^2 + 1)e^{i\theta } + |z_m|e^{i\phi _m} \]
\[ =  e^{i\theta }[2|z_m| \cos (\theta -\phi _m ) - (|z_m|^2 + 1)  ].\]
That is, $e^{-i\theta }\rho _m(e^{i\theta })$ is a real-valued nonvanishing function of $\theta \in \real $.
Similarly, $p_N(z)$ has no roots on the unit circle, and $C_1^{-1}e^{-iN\theta }p_N(e^{i\theta }) $ is real-valued for $\theta \in \real $.

\begin{thm} \label{eval1}
Suppose $L = \sum_{k=0}^N p_k(z)D^k$ is $\Dreg$ expression of order $N \ge 2$, whose leading coefficient $p_N(z)$ is $\Rsym$.
Let $\dop $ be the corresponding maximal operator.

For any $\epsilon $ satisfying $0 < \epsilon< 1$, each eigenvalue $\lambda $ is an element of a sequence  
$\{ \mu _n, n = 0 , \pm 1, \pm 2, \dots   \}$ having the form
\[\mu _n/C_1 =   ( n /\tau +C_2)^N  + O(n^{N-2 + \epsilon} ),\]
with constants $C_1\not= 0$, $\tau > 0$, and $C_2$.
The eigenvalues $\lambda $ of $\dop $ have algebraic multiplicity at most $2$ if $|\lambda |$ is sufficiently large.

\end{thm}

\begin{proof}

Since multiplication of $\dop $ by a constant $C_1$ simply multiplies eigenvalues by $C_1$,
it suffices to assume that $ p_N(z)$ has the form \eqref{pNform} with $C_1 = 1$.
The change of variables $z = e^{i\theta }$ with $\theta \in \real $ changes
$L$ to $L_{per} = \sum_{k=0}^N p_k(e^{i\theta })(-ie^{-i\theta } \frac{d}{d\theta })^k$,
with $L_{per}$ acting on $L^2_{per}$. 
If $\psi (z)$ is an eigenfunction of $\dop $, then $\psi (e^{i\theta })$ is a $2\pi$-periodic eigenfunction for $L_{per}$.
(The same remark applies to generalized eigenfunctions $\psi $ satisfying an equation $(\dop - \lambda I)^j\psi = 0$.)
The proof will proceed by showing that the sequence $\{ \mu _n \}$ describes 
the larger set of eigenvalues of $2\pi$-periodic eigenfunctions for $L_{per}$.

The leading coefficient $L_{per}$ is $(-i)^NP_N(\theta ) = (-i)^Ne^{-iN\theta }p_N(e^{i\theta })$.  As noted above, 
$P_N(\theta )$ is real-valued with no zeros for $\theta \in \real  $.  Absorbing the sign in $C_1$ if necessary,
assume that $P_N(\theta ) > 0$.  Conventional reductions \cite[p. 308-9]{CL} are available.  First 
use the change of real variables
\[t = \tau ^{-1}\int_0^{\theta } P_N(s)^{-1/N} \ ds , \quad \tau = \frac{1}{2\pi} \int_0^{2\pi } P_N(s)^{-1/N} \ ds ,\]
which carries $2\pi$-periodic functions of $\theta $ to $2\pi$-periodic functions of $t$.
Since $P_N(\theta )^{1/N}\ d/d \theta  = \tau ^{-1} d/dt$, the new expression has the form
\[L_t = (-i \tau ^{-1} d/dt)^N + \sum_{k=0}^{N-1} \tilde{p}_k(t)( -i\tau ^{-1}d/dt)^k,\] 
with $2\pi$-periodic coefficients $\tilde{p}_k(t)$. 

Next, let 
\[r(t) = \exp (\int_0^t - i\tau C + i\tau \frac{\tilde{p}_{N-1}(s)}{N} \ ds ), \quad C = \frac{1}{2\pi } \int_0^{2\pi }  \frac{\tilde{p}_{N-1}(s)}{N} \ ds .\]
The function $r(t)$ is periodic with period $2\pi $, and conjugation with $r(t)$ leaves an expression
\[L_{r} = r^{-1}(t)L_t r(t) = (-i \tau ^{-1}d/dt + C_2)^N +  \sum_{k=0}^{N-2} q_k(t)(-i\tau ^{-1} d/dt)^k, \quad C_2 = C/N .\]
The coefficients $q_k(t)$ are also periodic with period $2\pi $.  The eigenvalues of $2\pi $-periodic eigenfunctions 
for $L_{r}$ are the same as those for $L_{per}$.

The expression $(-i\tau ^{-1}d/dt )$ has $2\pi $-periodic eigenfunctions $\exp(i n t )$ with eigenvalues
$ n /\tau $, $n = 0,\pm 1, \pm 2, \dots $.  These eigenfunctions form an orthogonal basis for $L^2_{per}$.  Let $\dop _0$ denote the normal operator 
given by the expression $(-i \tau ^{-1}d/dt +C_2)^N $ with these eigenfunctions.  Let $\dop _p$ be given by the expression
$L_p =  \sum_{k=0}^{N-2} q_k(t)(-i \tau ^{-1}d/dt)^k$ on the domain of $\dop _0$, and take $\dop _{r} = \dop_0 + \dop _p$.
The eigenvalues of $\dop _0$ are $ \gamma _n = (n /\tau +C_2)^N $ whose behavior is described in \lemref{numstuff}. 
Suppose $0 < \epsilon < 1$, and $S_n$ is the circle $S_n = \{ \zeta \in \complex \ | \ |\zeta - \gamma _n | =  n^{N-2+\epsilon }\} $.
Then for $|n|$ sufficiently large, there is a $K > 0$ such that  
$|\zeta - \gamma _m | \ge Kn^{N-1}$ for all $\zeta \in S_n$ and $\gamma _m \not= \gamma _n$.

Let $R_r(\lambda ) = (\dop _{r} - \lambda I)^{-1}$ be the resolvent of $\dop _r$, and let $R_0(\lambda )$ be 
the resolvent of $\dop _0$.  Consider the formula 
\begin{equation} \label{reseqn}
(\dop _{r} - \lambda I)^{-1} = (\dop _0 + \dop _p - \lambda I)^{-1} = (\dop _0 -\lambda I )^{-1}(I +  \dop _p(\dop _0 - \lambda I)^{-1})^{-1} 
\end{equation}
with 
\[\dop _p(\dop _0 - \lambda I)^{-1} = \sum_{k=0}^{N-2} q_k (-i \tau ^{-1}d/dt)^k (\dop _0 - \lambda I)^{-1}.\]
Use an expansion $f = \sum_{m = -\infty }^{\infty} b_m\exp (imt)$ for $f \in L^2_{per}$ to compute 
\begin{equation} \label{rescalc}
 (-i \tau ^{-1} d/dt)^k (\dop _0 - \lambda I)^{-1} f =  \sum_{m= - \infty }^{\infty} b_m \frac{(m/\tau )^k }{\gamma _m - \lambda  } \exp (imt).
 \end{equation}
If $\lambda \in S_n$ and $ 0 \le k \le N-2$, then $\|(-i \tau ^{-1} d/dt)^k (\dop _0 - \lambda I)^{-1} \| = O(n^{-\epsilon }) $ as $|n| \to \infty $.
Similarly, since multiplication by $q_k$ is a bounded operator on $L^2_{per}$, $\| \dop _p(\dop _0 - \lambda I)^{-1} \| = O(n^{-\epsilon }) $  as $|n| \to \infty $ if $\lambda \in S_n$

Since $\dop _0$ is normal  \cite[p. 277]{Kato}, $\| R_0 (\lambda ) \| = \max_{\gamma _n} d(\lambda , \gamma _n)^{-1} $.
By \eqref{reseqn} and the calculation \eqref{rescalc}, if $| \lambda | $ is sufficiently large, then $\lambda $ is in the resolvent set for $\dop _r$ unless $\lambda $ lies inside some disk bounded by an $S_n$.
For $|n|$ sufficiently large and $\lambda \in S_n$,  \eqref{reseqn} gives
$\| R_r(\lambda ) - R_0(\lambda ) \| = O(n^{-N + 2 - 2 \epsilon}) $.
Recall \cite[ p. 67]{Kato} that the difference of the $\dop _r$ eigenprojections ${\cal P}_{r,n}$ and the $\dop _0$ eigenprojections ${\cal P}_{0,n}$ for eigenvalues inside $S_n$ is given by
\[ {\cal P}_{r,n} - {\cal P}_{0,n} =   -\frac{1}{2\pi i} \int_{S_n} R_r(\lambda ) - R_0(\lambda ) \ d \lambda .\] 
Since $\|  {\cal P}_{r,n} - {\cal P}_{0,n}  \| = O(n^{-\epsilon })$, 
the algebraic multiplicity of the eigenvalues of $\dop _r$ contained in $S_n$ is the same as for $\dop _0$, which is at most $2$, when
$|n|$ is large.  

\end{proof}

\thmref{eval1} does not distinguish between the eigenvalues coming from eigenfunctions of $L$ on $\Hardy _{\beta }$ and the larger set of eigenvalues
from $L^2_{per}$.  By working in $\Hardy _{\beta }$, the results of \thmref{eval1} can be refined.  The next lemma uses the expressions $B_{n,r}$ of 
\thmref{symmchar}.

\begin{lem} \label{symex}
Suppose $N \ge 1$ and the polynomial $p_N(z)$ is $\Rsym$.
For each $\sigma > 0$ there is a constant $C \not= 0$ and a formally symmetric expression 
\[L_0 = \sum_{r=0}^N [c_{N,r}B_{N,r} + \overline{c_{N,r}}B_{N,r}^+] \]
of order $N$ on $\Hardy _{\beta }$ whose leading coefficient is $Cp_N(z)$.
\end{lem}

\begin{proof}
As in \eqref{pNform}, the polynomial $p_N(z)$ has the form
\[p_N(z) = C_1\rho _1(z) \cdots \rho _M(z) z^{N-M}, \]
with each factor $\rho _m(z)$ of the form
\begin{equation} \label{rhoform}
\rho _m(z) = \overline{z_m}(z-z_m)(z - 1/\overline{z_m}) =  \overline{z_m}z^2 - (|z_m|^2 + 1)z + z_m .
\end{equation}

An induction proof will show that if $q(z) = \rho _1(z) \cdots \rho _M(z) $, then
\[q(z) = c_Mz^M + \sum_{j=0}^{M-1} [c_jz^j + \overline{c_j} z^{2M-j}], \quad c_M = \overline{c_M},\]
with the case of one factor established in \eqref{rhoform}.  Suppose the result holds for $q(z)$ with $M$ factors.
Then
\[ \rho _{M+1}(z) q(z) = [\overline{z_{M+1}}z^2 - (|z_{M+1}|^2 + 1)z + z_{M+1} ]\Bigl [c_Mz^M + \sum_{j=0}^{M-1} [c_jz^j + \overline{c_j} z^{2M-j}] \Bigr ].\]
After writng the polynomial  $\rho _{M+1}(z) q(z)$ in standard form, the coefficient of $z^{M+1}$ is $\overline{z_{M+1}}c_{M-1} - (|z_{M+1}|^2 + 1)c_M + z_{M+1}\overline{c_{M-1}} $, which is real.
For $j < M+1$ the coefficient of $z^{j}$ is $\overline{z_{M+1}}c_{j-2} - (|z_{M+1}|^2 + 1)c_{j-1} + z_{M+1}c_{j} $,
while the coefficient of $z^{2(M+1)-j}$ is $\overline{z_{M+1}}\overline{c_{j}} - (|z_{M+1}|^2 + 1)\overline{c_{j-1}} + z_{M+1}\overline{c_{j-2}} $,
preserving the desired symmetry.

Finally, the coefficients of $p_N(z) = \rho _1(z) \cdots \rho _M(z) z^{N-M}$ are obtained from those of  $\rho _1(z) \cdots \rho _M(z) $ by an index shift.
By \lemref{bsym} and \thmref{symmchar}, there is a  formally symmetric expression 
\begin{equation} \label{stanform}
L_0= \sum_{r=0}^N [c_{N,r}B_{N,r} + \overline{c_{N,r}}B_{N,r}^+] 
\end{equation}
of order $N$ on $\Hardy _{\beta }$ whose leading coefficient is $\rho _1(z) \cdots \rho _M(z) z^{N-M}$.
\end{proof}

Recall that if $L = \sum_{k=0}^N p_k(z)D^k$ is a $\Dreg$ formally symmetric expression of order $N \ge 1$ on $ \Hardy _{\beta }$,
then by \thmref{dregopsp} the domain of the maximal operator $\dop $ is $\Hardy _{\beta }^N$, and $\dop $ is self-adjoint
with compact resolvent by \thmref{regsa}.

\begin{thm} \label{eval2}
Suppose $L = \sum_{k=0}^N p_k(z)D^k$ is a $\Dreg$ formally symmetric expression of order $N \ge 2$ on $\Hardy _{\beta }$,
with self-adjoint maximal operator $\dop $.  The eigenvalues of $\dop $ can be enumerated as a sequence $\{ \lambda _n, n = 0,1,2,\dots \}$
with the following description: for $0 < \epsilon< 1$ there are real nonzero constants $C_1$ and $\tau > 0$, and a $C_2 \in \complex $ such that   
\[\lambda _n/C_1 =   ( n /\tau +C_2)^N  + O(n^{N-2 + \epsilon} ).\]
In particular,  with at most finitely many exceptions, the eigenvalues are either all positive or all negative, and have 
multiplicity $1$.  
\end{thm}

\begin{proof}

Let $z_1,\dots ,z_M$ be the nonzero roots of the leading coefficient $p_N(z)$ inside $\disk $. 
By \corref{symroots}, there are $M$ factors $\rho _m(z)$ as in \eqref{rhoform} such that
$p_N(z) = C_1\rho _1(z) \cdots \rho _M(z) z^{N-M}$.    Without loss of generality we may take $C_1 =1$.
The main idea is to construct and study a one parameter family of symmetric $\Dreg$ differential expressions $L(t)$ connecting $L$ 
with the elementary operator $\pm (zD)^N$.

For $0 \le t \le 1$ and each $\rho _m(z)$, define a one parameter family of polynomials by 
\[\rho _{m}(t,z) = t\overline{z_m}(z-tz_m)(z-\frac{1}{t\overline{z_m}}) = t \overline{z_m}z^2 - ( | tz_m|^2 + 1)z + tz_m .\]
These polynomials maintain the form \eqref{rhoform}, deforming $\rho _m(z)$ to $-z$ while keeping exactly one root inside $\disk $.  Now define 
\[p_N(t,z) = \rho _1(t,z) \cdots \rho _M(t,z) z^{N-M}.\]
By \eqref{stanform} of \lemref{symex} there is a one parameter family of formally symmetric expressions $L_0(t)$ with continuously varying coefficients and 
leading coefficient $p_N(t,z) = \rho _1(t,z) \cdots \rho _M(t,z) z^{N-M}$.  The expression $L - L_0(1)$ has order at most $N-1$.
By defining $L(t) = L_0(t) +  t(L - L_0(1))$
we obtain a one parameter family of formally symmetric expressions with continuously varying coefficients, with $L(0) = \pm (zD)^N$ and $L(1) = L$.
Let $\dop (t)$ denote the self-adjoint maximal operator with expression $L(t)$.

Fix $0 \le t_0 \le 1$, and let $\dop _p(t) = \dop (t) - \dop (t_0)$.
By \thmref{regsa}, the resolvent $R _0(\lambda ) = (\dop (t_0) - \lambda I)^{-1}$ is bounded 
from $\Hardy _{\beta }$ to $\Hardy _{\beta }^N$ uniformly on compact subsets of the resolvent set of $\dop (t_0)$.  
The perturbation formula for the resolvent $R(\lambda ,t)$,
\[ R(\lambda ,t)  = (\dop (t_0) + \dop _p(t) - \lambda I)^{-1} = (\dop (t_0) -\lambda I )^{-1}(I + \dop _p(t)(\dop (t_0) - \lambda I)^{-1})^{-1} \]
shows that $R(\lambda ,t): \Hardy _{\beta } \to \Hardy _{\beta }$ is a continuous operator valued function of $t$.
In particular the eigenvalues of $\dop (t)$ vary continuously with $t$.

Taking advantage of \thmref{eval1}, for each $0 \le t \le 1$  there is a sequence $ \gamma _n(t) = (n /\tau (t) +C_2(t))^N $
and circles $S_n(t) = \{ \zeta \in \complex \ | \ |\zeta - \gamma _n(t) | = n^{N-2+\epsilon} \} $ which contain all eigenvalues $\lambda _n(t)$
and are pairwise disjoint for $n$ sufficiently large.  Since $\gamma _n(t)$, the eigenvalues $\lambda _n(t)$,
and the resolvents are continuous functions, there is a cover of $[0,1]$ by open intervals $I_t$ centered at $t$ such that
for $t_1 \in I_t$ all eigenvalues of $\dop (t_1)$ are contained in some $S_n(t)$, with $S_n(t)$ containing the same 
number of eigenvalues for $\dop (t)$ and $\dop (t_1)$ if $n$ is sufficiently large.
Taking a finite subcover, and using the fact that $\dop (0)$ has expression $(zD)^N$, with eigenvalues $n^N$ for $n=0,1,2,\dots $, 
it follows that for $n$ large there is one eigenvalue $\lambda _n$, counted with multiplicity, in $S_n(1)$.

\end{proof}

\bibliographystyle{amsalpha}

\end{document}